\documentclass[oneside,12pt]{amsart}
\usepackage{amsmath, amsfonts,amsthm,times,graphics}

 \makeatletter
\renewcommand*\subjclass[2][2010]{%
  \def\@subjclass{#2}%
  \@ifundefined{subjclassname@#1}{%
    \ClassWarning{\@classname}{Unknown edition (#1) of Mathematics
      Subject Classification; using '2010'.}%
  }{%
    \@xp\let\@xp\subjclassname\csname subjclassname@#1\endcsname
  }%
}

 \makeatother

\newtheorem{lemma}{Lemma}
\newtheorem{corollary}{Corollary}
\newtheorem{remark}{Remark}
\newtheorem{proposition}{Proposition}
\theoremstyle{definition}

 \makeatletter
\renewcommand*\subjclass[2][2010]{%
  \def\@subjclass{#2}%
  \@ifundefined{subjclassname@#1}{%
    \ClassWarning{\@classname}{Unknown edition (#1) of Mathematics
      Subject Classification; using '1991'.}%
  }{%
    \@xp\let\@xp\subjclassname\csname subjclassname@#1\endcsname
  }%
}
 \makeatother

\begin{document}

\title{Congruences for Wolstenholme primes}

\author{Romeo Me\v strovi\' c}

\address{Maritime Faculty, University of Montenegro, Dobrota 36,
 85330 Kotor, Montenegro} \email{romeo@ac.me}

{\renewcommand{\thefootnote}{}
\setcounter{footnote}{0}}

\maketitle

\begin{abstract}
A prime number $p$ is said to be  a Wolstenholme prime if it 
satisfies the congruence ${2p-1\choose p-1} \equiv 1 \,\,(\bmod{\,\,p^4})$.
For such a prime $p$, we establish  the expression for 
${2p-1\choose p-1}\,\,(\bmod{\,\,p^8})$
given in terms of the sums $R_i:=\sum_{k=1}^{p-1}1/k^i$ ($i=1,2,3,4,5,6)$.
Further, the expression in this congruence is 
reduced in terms of the sums $R_i$ ($i=1,3,4,5$).
Using this congruence, we prove  that for any Wolstenholme prime, 
$$
{2p-1\choose p-1}\equiv 1 -2p \sum_{k=1}^{p-1}\frac{1}{k}
-2p^2\sum_{k=1}^{p-1}\frac{1}{k^2}\pmod{p^7}. 
 $$
Moreover, using a recent result of the author \cite{Me}, we prove 
that the above congruence implies that a prime $p$ 
necessarily must be a Wolstenholme prime.

Applying a technique of Helou and Terjanian \cite{HT}, 
the above congruence is given as the 
expression involving the Bernoulli numbers. 
 \end{abstract}

\section{Introduction and Statements of  Results}

{\it Wolstenholme's theorem} (e.g., see \cite{W}, \cite{Gr})
asserts that if $p$ is a prime greater than 3, then 
the binomial coefficient
${2p-1\choose p-1}$ satisfies the congruence
 $$
{2p-1\choose p-1} \equiv 1 \pmod{p^3}.
   $$
It is well known (e.g., see \cite{HW}) that
this theorem is equivalent  to the 
assertion that for any prime $p\ge 5$ the numerator of the 
fraction 
 $$
1+\frac{1}{2}+\frac{1}{3}+\cdots+\frac{1}{p-1}
 $$ 
written in reduced form, is divisible by $p^2$. 
A. Granwille \cite{Gr} established broader generalizations of 
Wolstenholme's theorem. As an application, it is obtained in \cite{Gr}
that for a prime $p\ge 5$ there holds
  $$
{2p-1\choose p-1}\Big/{2p\choose p}^3\equiv
{3\choose 2}\Big/{2\choose 1}^3\pmod{p^5}.
  $$
Recently, by studying  Fleck's quotients,
Z.W. Sun and D. Wan (\cite[Corollary 1.5]{SW}) discovered a new extension of 
Wolstenholme's congruences. In particular, their result 
yields Wolstenholme's theorem and for a prime $p\ge 7$ the following 
new curious congruence
   $$
{4p-1\choose 2p-1} \equiv {4p\choose p}-1 \pmod{p^5}.
   $$
More recently, C. Helou and G. Terjanian \cite{HT} established many  
Wolstenholme's type congruences modulo $p^k$ with a prime $p$
and $k\in\mathbf N$ such that $k\le 6$. One of their main results 
(\cite[Proposition 2, pp. 488-489]{HT})
is a congruence off the form ${np\choose mp}\equiv f(n,m,p) {n\choose m}
\,(\bmod\,\, p)$,
where $p\ge 3$ is a prime number, $m,n,\in \mathbf N$ with 
$0\le m\le n$, and $f$ is the function on $m,n$ and $p$ involving 
Bernoulli numbers $B_k$. 
As an application, by (\cite[ Corollary 2(2), p. 493; also see Corollary 
6(2), p. 495)]{HT}, for any prime $p\ge 5$ we have 
$$  
{2p-1\choose p-1} \equiv 1-p^3B_{p^3-p^2-2}
+\frac{1}{3}p^5B_{p-3}-\frac{6}{5}p^5B_{p-5}\pmod{p^6}.
 $$
A similar congruence  modulo $p^7$ (Corollary~\ref{c2}) is obtained
in this paper for Wolstenholme primes.

A prime $p$ is said to be a {\it Wolstenholme prime} if it 
satisfies the congruence 
$$
{2p-1\choose p-1} \equiv 1 \,\,(\bmod{\,\,p^4}).
 $$
The two known such primes are 16843 and 2124679, and 
recently, R.J. McIntosh and E.L. Roettger \cite{MR} reported  that these
primes are only two  Wolstenholme primes less than $10^9$.
However, using the argument based on the prime number theorem, 
McIntosh  (\cite[p. 387]{M}) conjectured that there are infinitely many 
Wolstenholme primes, and that no  prime satisfies
the congruence ${2p-1\choose p-1} \equiv 1\,\, (\bmod{\,\,p^5})$.

The following result is basic in our investigations. 

\begin{proposition}\label{p1}
 Let $p$ be a Wolstenholme prime. Then
   \begin{eqnarray*}  
{2p-1\choose p-1} &\equiv& 1 +p \sum_{k=1}^{p-1}\frac{1}{k}
-\frac{p^2}{2} \sum_{k=1}^{p-1}\frac{1}{k^2}+
\frac{p^3}{3} \sum_{k=1}^{p-1}\frac{1}{k^3}-
\frac{p^4}{4} \sum_{k=1}^{p-1}\frac{1}{k^4}\\
&&+\frac{p^5}{5} \sum_{k=1}^{p-1}\frac{1}{k^5}-
\frac{p^6}{6} \sum_{k=1}^{p-1}\frac{1}{k^6}\pmod{p^8}.
   \end{eqnarray*}
\end{proposition}  
The above congruence can be simplified as follows.

\begin{proposition}\label{p2}
 Let $p$ be a Wolstenholme prime. Then
  $$  
{2p-1\choose p-1} \equiv 1 +\frac{3p}{2} \sum_{k=1}^{p-1}\frac{1}{k}
-\frac{p^2}{4}\sum_{k=1}^{p-1}\frac{1}{k^2}
+\frac{7p^3}{12} \sum_{k=1}^{p-1}\frac{1}{k^3}
+\frac{5p^5}{12} \sum_{k=1}^{p-1}\frac{1}{k^5}\pmod{p^8}.
   $$  
\end{proposition}
Reducing the modulus in the previous congruence,
we can obtain the following simpler congruences.

\begin{corollary}\label{c1}
 Let $p$ be a Wolstenholme prime. Then
   $$  
{2p-1\choose p-1} \equiv 1 -2p \sum_{k=1}^{p-1}\frac{1}{k}
-2p^2\sum_{k=1}^{p-1}\frac{1}{k^2}
 \equiv 1 +2p \sum_{k=1}^{p-1}\frac{1}{k}
+\frac{2p^3}{3}\sum_{k=1}^{p-1}\frac{1}{k^3}\pmod{p^7}.
    $$  
\end{corollary}

The {\it Bernoulli numbers} $B_k$ ($k\in\mathbf N$) 
are defined by the generating function
   $$
\sum_{k=0}^{\infty}B_k\frac{x^k}{k!}=\frac{x}{e^x-1}\,.
  $$
It is easy to find the values $B_0=1$, $B_1=-\frac{1}{2}$, 
$B_2=\frac{1}{6}$, $B_4=-\frac{1}{30}$, and $B_n=0$ for odd $n\ge 3$. 
Furthermore, $(-1)^{n-1}B_{2n}>0$ for all $n\ge 1$. 
These and many other properties can be found, for instance, in \cite{IR}. 

The congruence from the Corollary~\ref{c1} is given in terms of  the Bernoulli numbers
by the following result.

\begin{corollary}\label{c2}
 Let $p$ be a Wolstenholme prime. Then
 $$  
{2p-1\choose p-1} \equiv 1-p^3B_{p^4-p^3-2}
-\frac{3}{2}p^5B_{p^2-p-4}+\frac{3}{10}p^6B_{p-5}\pmod{p^7}.
 $$
\end{corollary}

The above congruence can be given by the following 
expression involving lower order Bernoulli numbers.

\begin{corollary}\label{c3}
 Let $p$ be a Wolstenholme prime. Then
        \begin{eqnarray*}    
&&{2p-1\choose p-1} \equiv 1-
p^3\left(\frac{8}{3}B_{p-3}
-3B_{2p-4}+\frac{8}{5}B_{3p-5}-\frac{1}{3}B_{4p-6}\right)-
p^4\left(\frac{8}{9}B_{p-3}\right.\\
&&-\frac{3}{2}B_{2p-4}+\left.\frac{24}{25}B_{3p-5}-\frac{2}{9}B_{4p-6}\right)
-p^5\left(\frac{8}{27}B_{p-3}-\frac{3}{4}B_{2p-4}+\frac{72}{125}B_{3p-5}
\right.\\
&&-\frac{4}{27}B_{4p-6}+\left.\frac{12}{5}B_{p-5}-B_{2p-6}\right)
-\frac{2}{25}p^6B_{p-5}\pmod{p^7}.
 \end{eqnarray*}    
\end{corollary}
Combining the first congruence in Proposition~\ref{p1} and a recent
result of the author in \cite[Theorem 1.1]{Me}, 
we obtain a new characterization of Wolstenholme prime 
as follows. 

\begin{corollary}\label{c4} {\rm (\cite[Remark 1.6]{Me})}.
 A prime  $p$ is  a Wolstenholme prime if and only if
  $$  
{2p-1\choose p-1} \equiv 1 -2p \sum_{k=1}^{p-1}\frac{1}{k}
-2p^2\sum_{k=1}^{p-1}\frac{1}{k^2}
 \pmod{p^7}.
   $$
\end{corollary}
 
\begin{remark} {\rm A computation shows that no prime  $p<10^5$
satisfies the second congruence from Corollary~\ref{c1}, except the Wolstenholme 
prime $16843$.
Accordingly, an interesting question is as follows:
{\it Is it  true that the second congruence  from Corollary~{\rm\ref{c1}} 
yields that a prime $p$ is necessarily 
a Wolstenholme prime}?
We conjecture that this is true.}
 \end{remark}
\section{Proof of Proposition~\ref{p1} }

For the proof of Proposition~\ref{p1}, we will need some auxiliary results.

\begin{lemma}\label{l1}
 For any prime $p\ge 7$, we have
  \begin{equation}\label{(1)}
2\sum_{k=1}^{p-1}\frac{1}{k}\equiv
-p\sum_{k=1}^{p-1}\frac{1}{k^2}\pmod{p^4}
  \end{equation}
\end{lemma}
\begin{proof}
The above congruence is in fact, the congruence 
(14) in  (\cite[ Proof of Theorem 3.2]{Z2}). 
\end{proof}

\begin{lemma}\label{l2} 
For any prime $p\ge 7$, we have
       \begin{equation}\label{(2)}   
{2p-1\choose p-1} \equiv 1 +2p \sum_{k=1}^{p-1}\frac{1}{k}\pmod{p^5},
  \end{equation}
and 
     \begin{equation}\label{(3)}
{2p-1\choose p-1} \equiv 1 - p^2 \sum_{k=1}^{p-1}\frac{1}{k^2}\pmod{p^5}.
  \end{equation}   
\end{lemma}
\begin{proof} 
 Let $R_1(p)=\sum_{k=1}^{p-1}1/k$.
Following  (\cite[Definition 3.1]{Z2}) we define $w_p<p^2$ to be the unique 
nonnegative  integer such that $w_p\equiv  R_1(p)/p^2\,\,(\bmod{\,\,p^2})$.
Then by (\cite[ Theorem 3.2]{Z2}), for all nonnegative integers $n$ and $r$
with $n\ge r$,
  \begin{equation}\label{(4)} 
{np\choose rp}\bigg/{n\choose r} \equiv 1 +w_p nr(n-r)p^3\pmod{p^5}.
  \end{equation}
Since $\frac 12{2p\choose p}={2p-1\choose p-1}$, 
taking $n=2$ and $r=1$ into  (\ref{(4)}), it becomes 
 $$ 
{2p-1\choose p-1}\equiv 1 +2w_pp^3 \pmod{p^5},
  $$
which is actually (\ref{(2)}). Now the congruence (\ref{(3)}) 
follows immediately from (\ref{(2)}) and (\ref{(1)}) of Lemma~\ref{l1}.
\end{proof}

\begin{lemma}\label{l3}
 The following statements about a prime
$p\ge 7$ are equivalent:
\begin{itemize}
\item[(i)] $p$  is a Wolstenholme prime;\\\indent
\item[(ii) ]
 $\displaystyle\sum_{k=1}^{p-1}\frac{1}{k}\equiv
0\pmod{p^3};$\\
\item[(iii)] 
$\displaystyle\sum_{k=1}^{p-1}\frac{1}{k^2}\equiv
0\pmod{p^2}$; \\\indent
\item[(iv)] $p$ divides the numerator of the Bernoulli number 
$B_{p-3}.$
\end{itemize}
\end{lemma}
\begin{proof}
 The equivalences (i)$\Leftrightarrow$(ii)$\Leftrightarrow$(iii)
are immediate from Lemma~\ref{l2} if we consider the congruences (\ref{(2)}) 
and (\ref{(3)})  modulo $p^4$. Further, by a special case of 
Glaisher's congruence (\cite[p. 21]{Gl1}, \cite[ p. 323]{Gl2};   
also cf. \cite[  Theorem 2]{M}), we have
   $$
{2p-1\choose p-1}\equiv 1 -\frac{2}{3}p^3B_{p-3}\pmod{p^4},
   $$
whence follows  the equivalence 
(i)$\Leftrightarrow$(iv). This concludes the proof. 
\end{proof}

\begin{remark} {\rm For the proof of Proposition~\ref{p1},
we use the congruences (\ref{(2)}) and (\ref{(3)}) of Lemma~\ref{l2}
with $(\bmod{\,\,p^4})$ instead of $(\bmod{\,\,p^5})$.
By a classical result of E. Lehmer
(\cite{L}; also see \cite[ Theorem 2.8]{Z1}),
$\sum_{k=1}^{p-1}1/k\equiv -\frac{p^2}{3}B_{p-3}\,\,(\bmod{\,\,p^3})$.
Substituting this into the Glaisher's congruence given above, 
we obtain immediately 
(\ref{(2)}) of Lemma~\ref{l2}, with $(\bmod{\,\,p^4})$
instead of $(\bmod{\,\,p^5})$.

Note that the congruence (\ref{(3)}) is also given in (\cite[ p. 385]{M}),
but its  proof is there omitted.} 
\end{remark}

For a prime $p\ge 3$  and a positive integer
$n\le p-2$  we denote 
  $$
R_n(p):=\sum_{i=1}^{p-1}\frac{1}{k^n}\quad
{\rm and}\quad
H_n(p):=\sum_{1\le i_1<i_2<\cdots<i_n\le p-1}\frac{1}{i_1i_2\cdots i_n}.
  $$
Recall that in the sequel we shall often write
throughout proofs  $R_n$ and $H_n$ instead of 
$R_n(p)$ and $H_n(p)$, respectively.

\begin{lemma}\label{l4} {\rm(\cite[Theorem 3]{B}; also 
see   \cite[Remark 2.3]{Z1})}.  {\it For any
prime $p\ge 3$  and a positive integer $n\le p-3$, we have} 
 $$
R_n(p)\equiv 0  \pmod{p^2}\,\,
if\,\,  n\,\, is\,\, odd ,\,\,  and\,\, R_n(p)\equiv 0  \pmod{p}\quad
if\quad n\,\, is\,\, even.
 $$ 
\end{lemma}

\begin{lemma}\label{l5} {\rm(Newton's formula, see e.g., \cite{J})}. 
Let $m$ and $s$ be positive integers such that $m\le s$. Define the
symmetric polynomials 
 $$
P_m(s)=P_m(s;x_1,x_2,\ldots , x_s)=x_1^m+x_2^m+\cdots +x_s^m,
 $$
and
$$
A_m(s)=A_m(s;x_1,x_2,\ldots , x_s)=\sum_{1\le i_1<i_2<\cdots<i_m
\le s}x{_{i_1}}x{_{i_2}}\cdots x{_{i_m}}.
 $$ 
Then for $n=1,2,\ldots,s$, we have
 $$
P_n(s)-A_1(s)P_{n-1}(s)+A_2(s)P_{n-2}(s)+\cdots +(-1)^{n-1}A_{n-1}(s)P_1(s)+
(-1)^nnA_n(s)=0.
 $$
\end{lemma}

\begin{lemma}\label{l6} 
For any prime $p\ge 5$  and a positive integer
$n\le p-2$,  we have 
 $$
H_n(p)\equiv 0  \pmod{p^2}\,\,
if\,\, n\,\, is \,\, odd \quad and\quad H_n(p)\equiv 0  \pmod{p}\,\,
if\,\,  n\,\, is\,\, even.
 $$  
\end{lemma}
\begin{proof}
 According to the notations of Lemma~\ref{l5}, 
setting $s=p-1$ and $x_k=1/k$ with $k=1,2,\ldots,p-1$, 
for $n=1,2,\ldots,p-1$, we have
      $$
 P_n\left(p-1;1,\frac{1}{2},\ldots, \frac{1}{p-1}\right)=
\sum_{i=1}^{p-1}\frac{1}{k^n}=R_n(p),
  $$
and 
 $$
 A_n\left(p-1;1,\frac{1}{2},\ldots , \frac{1}{p-1}\right)=
\sum_{1\le i_1<i_2<\cdots<i_n\le p-1}\frac{1}{i_1i_2\cdots i_n}=H_n(p).
 $$
Then by Newton's formula  (see Lemma~\ref{l5}),  we have
 \begin{equation}\label{(5)}
R_n-H_1R_{n-1}+H_2R_{n-2}+\cdots +(-1)^{n-1}H_{n-1}R_1+
(-1)^nnH_n=0.
  \end{equation}
Therefore
  \begin{equation}\label{(6)}
H_n=\frac{(-1)^{n-1}}{n}\left(R_n+\sum_{i=1}^{n-1}(-1)^iH_iR_{n-i}\right)
\quad {\rm for}\quad  n=1,2,\ldots, p-2.
  \end{equation}
We proceed by induction on $n$. If $n=1$, then
by Wolstenholme's theorem,
$$
H_1=R_1\equiv 0  \,\,(\bmod{\,\,p^2}).
$$
Now suppose that for a fixed $n-1$ with $1\le n-1\le p-3$
and for each $i$ with $1\le i\le n-1$ holds
 $$
H_i\equiv 0  \,\,(\bmod{\,\,p^2})\quad
{\rm if}\,\, i\,\, {\rm is \,\, odd}, \quad 
{\rm and} \quad H_i\equiv 0  \,\,(\bmod{\,\,p})
\quad
{\rm if}\,\, i\,\, {\rm is\,\, even}. 
 $$  
From this assumption and Lemma~\ref{l4} it follows that 
  \begin{equation}\label{(7)}
H_iR_{n-i}\equiv 0\,\,(\bmod{\,\,p^2})\quad{\rm for\,\, all}\quad
i=1,2,\ldots,n-1.
  \end{equation}
If $n$ is odd, then by Lemma~\ref{l4}, $p^2\mid R_n$.  Substituting this
and congruences (\ref{(7)}) into (\ref{(6)}), we obtain
$H_n\equiv 0\,\,(\bmod{\,\,p^2})$.

Similarly, if $n$ is even, then by Lemma~\ref{l4}, $p^2\mid R_n$. 
This together with (\ref{(7)}) and (\ref{(6)}) yields 
$H_n\equiv 0\,\,(\bmod{\,\,p})$.
This concludes the induction  proof.
\end{proof}

\begin{remark} {\rm Note that Lemma~\ref{l6} is an immediate
consequence of a recent result
of X. Zhou and T. Cai (\cite[Lemma 2]{ZC}; also 
see \cite[Theorem 2.14]{Z1}).  }
\end{remark}

\begin{lemma}\label{l7} 
For any Wolstenholme prime $p$, 
we have 
    $$    
R_2(p)\equiv -2H_2(p) \pmod{p^6},\,\,
R_3(p)\equiv 3H_3(p)\pmod{p^5},
 $$
  $$
 R_4(p)\equiv -4H_4(p)\pmod{p^4},\,\, R_5(p)\equiv 5H_5(p)\pmod{p^4}
and
  $$
$$
R_6(p)\equiv -6H_6(p) \pmod{p^3}.
 $$
  \end{lemma}
\begin{proof}  We use 
the formula (\ref{(5)}) for $n=2,3,4,5,6$ in the form
 \begin{equation}\label{(8)}
R_n+(-1)^nnH_n=H_1R_{n-1}-H_2R_{n-2}+\cdots +(-1)^{n}H_{n-1}R_1.
  \end{equation}
First note that by Lemma~\ref{l3}, 
$R_1=H_1\equiv 0 \,\,(\bmod{\,\, p^3})$ and
$R_2\equiv 0 \,\,(\bmod{\,\, p^2})$
Therefore, (\ref{(8)}) implies $R_2+2H_2=H_1R_1\equiv 0 \,\,(\bmod{\,\, p^6})$,
so that, $R_2\equiv -2H_2\,\,(\bmod{\,\, p^6})$.
From this and Lemma~\ref{l3} we conclude that
$H_2\equiv R_2\equiv 0 \,\,(\bmod{\,\, p^2})$.

Further, by Lemma~\ref{l4} and Lemma~\ref{l6}, $R_3\equiv H_3\equiv R_5\equiv H_5\equiv 0 
\,\,(\bmod{\,\, p^2})$ and $R_4\equiv H_4\equiv 0 \,\,(\bmod{\,\, p})$.
Substituting the previous congruences 
for $H_i$ and $R_i$ $(i=1,2,3,4,5)$
into (\ref{(8)}) with $n=3,4,5,6$, we get
     \begin{eqnarray*}  
R_3-3H_3&=&H_1R_2-H_2R_1\equiv 0 \pmod{p^5},\\
 R_4+4H_4&=&H_1R_3-H_2R_2+H_3R_1\equiv 0 \pmod{p^4},\\
   R_5-5H_5&=&H_1R_4-H_2R_3+H_3R_2-H_4R_1\equiv 0 \pmod{p^4},\\
R_6+6H_6&=&H_1R_5-H_2R_4+H_3R_3-H_4R_2+H_5R_1\equiv 0 \pmod{p^3}.
     \end{eqnarray*}  
This completes the proof.
\end{proof}

\begin{proof}[Proof of  Proposition~{\rm\ref{p1}}] For any prime $p\ge 7$, we have
\begin{eqnarray*}   
{2p-1\choose p-1}&=&\frac{(p+1)(p+2)\cdots(p+k)\cdots (p+(p-1))}{1\cdot 2
\cdots k\cdots p-1}\\
&=& \left(\frac{p}{1}+1\right)\left(\frac{p}{2}+1\right)\cdots\left(\frac{p}{k}+1\right)
\cdots\left(\frac{p}{p-1}+1\right)\\
&=&1+\sum_{i=1}^{p-1}\frac{p}{i}+
\sum_{1\le i_1<i_2\le p-1}\frac{p^2}{i_1i_2}+\cdots+
\sum_{1\le i_1<i_2<\cdots <i_k\le p-1}\frac{p^k}{i_1i_2\cdots i_k}\\
&&+\cdots+\frac{p^{p-1}}{(p-1)!}=1+\sum_{k=1}^{p-1}p^{k}H_k=
1+\sum_{k=1}^{6}p^{k}H_k+\sum_{k=7}^{p-1}p^{k}H_k.
 \end{eqnarray*}   
Since by Lemma~\ref{l6}, $p^9\mid \sum_{k=7}^{p-1}p^{k}H_k$ 
for any prime $p\ge 11$, 
the above identity yields
 $$
{2p-1\choose p-1}\equiv 1+pH_1+p^2H_2+p^3H_3+p^4H_4
+p^5H_5+p^6H_6\pmod{p^8}.
 $$
Now by Lemma~\ref{l7}, for $n=2,3,4,5,6$, we have
 $$
H_n\equiv(-1)^{n-1}\frac{R_n}{n}\pmod{p^{e_n}} \,\,{\rm for}\,\,
e_2=6, e_3=5, e_4=4, e_5=4\,\, {\rm and}\,\, e_6=3.
 $$
Substituting the above congruences into the previous one, 
and setting $H_1=R_1$, we obtain
 $$
{2p-1\choose p-1}\equiv 1+ pR_1-\frac{p^2}{2}R_2+
\frac{p^3}{3}R_3-\frac{p^4}{4}R_4+
\frac{p^5}{5}R_5-\frac{p^6}{6}R_6\pmod{p^8}.
 $$
This is the desired congruence from Proposition~\ref{p1}. 
\end{proof}

\section{Proofs of  Proposition~\ref{p2} and Corollaries~\ref{c1}--\ref{c3}}

In order to prove  Proposition~\ref{p2} and Corollaries~\ref{c1}--\ref{c3},
we need some auxiliary results.

\begin{lemma}\label{l8} Let  $p$  be a prime, 
and let $m$ be any even positive integer.
Then the denominator $d_{m}$ of the Bernoulli number $B_{m}$
written in reduced form,  is given by
   $$
d_{m}=\prod_{p-1\mid m}p,
  $$
where the product is taken over those primes $p$ such that $p-1$ divides $m$. 
 \end{lemma}
  \begin{proof}  The assertion is an immediate 
consequence of the von Staudt-Clausen theorem 
(e.g., see \cite[ p. 233, Theorem 3]{IR}) which asserts
that  $B_{m}+\sum_{p-1\mid m}1/p$ is an integer for all even $m$, where 
the summation is over all primes $p$ such that 
$p-1$ divides  $m$.   
\end{proof}

Recall that for a prime $p$  and a positive integer
$n$,  we denote 
  $$
R_n(p)=R_n=\sum_{k=1}^{p-1}\frac{1}{k^n}\quad
{\rm and}\quad P_n(p)=\sum_{k=1}^{p-1}k^n.
  $$
\begin{lemma}\label{l9} {\rm(\cite[p. 8]{HT})}. 
Let $p$ be a prime greater than
$5$, and let $n,r$ be positive integers. Then
  \begin{equation}\label{(9)}
P_n(p)\equiv\sum_{s-{\rm ord}_p(s)\le r}
\frac{1}{s}{n\choose s-1}p^sB_{n+1-s}\pmod{p^r},
  \end{equation}
where ${\rm ord}_p(s)$ is the largest power of $p$ dividing $s$,
and the summation is taken over all integers $1\le s\le n+1$
such that $s-{\rm ord}_p(s)\le r$.
  \end{lemma}
The following result is well known as the Kummer congruences.

\begin{lemma}\label{l10} {\rm(\cite{IR})}. 
Suppose that $p\ge 3$ is a prime and 
$m$, $n$, $r$ are positive integers such that $m$ and $n$ are even,
$r\le n-1\le m-1$ and $m\not\equiv 0\,\,(\bmod{\,\, p-1})$.
If $n\equiv m\,\,(\bmod{\,\,\varphi(p^r)})$, where $\varphi(p^r)=p^{r-1}(p-1)$
is the Euler's totient function, then
 \begin{equation}\label{(10)}
\frac{B_m}{m}\equiv\frac{B_n}{n}\pmod{p^r}.
  \end{equation}
 \end{lemma}

The following congruences are also due to Kummer.

\begin{lemma}\label{l11} {\rm(\cite{K}; also see \cite[p. 20]{HT})}. 
 Let $p\ge 3$ 
be a prime and let $m$, $r$ be positive integers such that $m$ is even, 
$r\le m-1$ and $m\not\equiv 0\,\,(\bmod{\,\, p-1})$. Then
 \begin{equation}\label{(11)}
\sum_{k=0}^{r}(-1)^k{m\choose k}\frac{B_{m+k(p-1)}}{m+k(p-1)}\equiv 0
\pmod{p^r}.
  \end{equation}
\end{lemma}

\begin{lemma}\label{l12}  For any prime $p\ge 11$, we have
\begin{itemize}
\item[(i)] $\displaystyle R_1(p)\equiv -\frac{p^2}{2}B_{p^4-p^3-2}
-\frac{p^4}{4}B_{p^2-p-4}+\frac{p^5}{6}B_{p-3}+\frac{p^5}{20}B_{p-5}
\,\,(\bmod{\,\, p^6}).\qquad$
\item[(ii)] $\displaystyle R_3(p)\equiv 
-\frac{3}{2}p^2B_{p^4-p^3-4}\,\,(\bmod{\,\, p^4}).$ 
\item[(iii)] $\displaystyle R_4(p)\equiv pB_{p^4-p^3-4}\,\,(\bmod{\,\, p^3}).$

\item[(iv)] $\displaystyle pR_6(p)\equiv -\frac{2}{5}R_5(p)\,\,(\bmod{\,\, p^4}).$
\end{itemize}
\end{lemma}
\begin{proof}
 If $s$ is a positive integer such that 
${\rm ord}_p(s)=e\ge 1$, then for $p\ge 11$ holds $s-e\ge p^e-e\ge 10$.
This shows that the condition $s-{\rm ord}_p(s)\le 6$
implies that ${\rm ord}_p(s)=0$, and thus, for such a $s$ must be $s\le 6$.
Therefore
  \begin{equation}\label{(12)}
P_n(p)\equiv\sum_{s=1}^6\frac{1}{s}{n\choose s-1}
p^sB_{n+1-s}\pmod{p^6}\quad {\rm for}\quad n=1,2,\ldots. 
  \end{equation}
By Euler's theorem \cite{HW}, for $1\le k\le p-1$, and positive integers
$n,e$ we have $1/k^{\varphi(p^e)-n}\equiv k^{n}\,\,(\bmod{\,\, p^e})$,
where   $\varphi(p^e)=p^{e-1}(p-1)$ is the Euler's totient function. 
Hence, $R_{\varphi(p^e)-n}(p)\equiv P_n(p)\,\,(\bmod{\,\, p^e})$.
In particular, if $n=\varphi(p^6)-1=p^5(p-1)-1$, then 
by Lemma~\ref{l8}, $p^6\mid p^6B_{p^5(p-1)-6}$ for each prime $p\ge 11$.
Therefore, using the fact that 
$B_{p^5(p-1)-1}=B_{p^5(p-1)-3}=B_{p^5(p-1)-5}=0$, (\ref{(12)}) yields 
    \begin{eqnarray*}   
R_1(p)&\equiv & P_{p^5(p-1)-1}(p)\equiv \frac{1}{2}(p^5(p-1)-1)p^2
B_{p^5(p-1)-2}\\
&& +\frac{1}{4}\frac{(p^5(p-1)-1)(p^5(p-1)-2)(p^5(p-1)-3)}{6}p^4
B_{p^5(p-1)-4}\pmod{p^6},
       \end{eqnarray*}  
whence we have
  \begin{equation}\label{(13)}
R_1(p)\equiv -\frac{p^2}{2}B_{p^6-p^5-2}
-\frac{p^4}{4}B_{p^6-p^5-4}\pmod{p^6}.
  \end{equation}
By the Kummer congruences (\ref{(10)}) from Lemma~\ref{l10}, we have
 $$
B_{p^6-p^5-2}\equiv
\frac{p^6-p^5-2}{p^4-p^3-2}B_{p^4-p^3-2}\equiv
\frac{2B_{p^4-p^3-2}}{p^3+2}\equiv
\left(1-\frac{p^3}{2}\right)B_{p^4-p^3-2}\pmod{p^4}.
 $$
Substituting this into (\ref{(13)}), we obtain
 \begin{equation}\label{(14)}
R_1(p)\equiv -\frac{p^2}{2}B_{p^4-p^3-2}+\frac{p^5}{4}B_{p^4-p^3-2}
-\frac{p^4}{4}B_{p^6-p^5-4}\pmod{p^6}.
 \end{equation}
Similarly, we have 
 $$
B_{p^4-p^3-2}\equiv
\frac{p^4-p^3-2}{p-3}B_{p-3}\equiv \frac{2}{3}B_{p-3}\pmod{p}
 $$
and
 $$
B_{p^6-p^5-4}\equiv
\frac{p^6-p^5-4}{p^2-p-4}B_{p^2-p-4}\equiv \frac{4B_{p^2-p-4}}{p+4}
\equiv \left(1-\frac{p}{4}\right)B_{p^2-p-4}\pmod{p^2}.
 $$
Substituting the above two congruences into (\ref{(14)}), we get
 \begin{equation}\label{(15)}
R_1(p)\equiv -\frac{p^2}{2}B_{p^4-p^3-2}+\frac{p^5}{6}B_{p-3}
-\frac{p^4}{4}B_{p^2-p-4}+\frac{p^5}{16}B_{p^2-p-4}\pmod{p^6}.
  \end{equation}
Finally, since
$$
B_{p^2-p-4}\equiv\frac{p^2-p-4}{p-5}B_{p-5}\equiv
\frac{4}{5}B_{p-5}\pmod{p},
 $$
the substitution of the above 
congruence  into (\ref{(15)}) immediately gives the congruence (i).
 
To prove the congruences (ii) and (iii),
note that if $n-3\not\equiv 0\,\,(\bmod{\,\, p-1})$,
then by Lemma~\ref{l8}, $p^4\mid p^4B_{n-3}$
for odd $n\ge 5$, while $B_{n-3}=0$ for even $n\ge 6$.
Therefore, reducing the modulus in (\ref{(12)}) to  $p^4$, for 
all odd $n\ge 3$   with $n-3\not\equiv 0\,\,(\bmod{\,\, p-1})$
and for all even $n\ge 2$  holds
 \begin{equation}\label{(16)}
P_n(p)\equiv pB_n+\frac{p^2}{2}nB_{n-1}+\frac{p^3}{6}n(n-1)B_{n-2}\pmod{p^4}.
  \end{equation}
In particular, for $n=p^4-p^3-3$ we have
$B_{p^4-p^3-3}=B_{p^4-p^3-5}=0$, and thus (\ref{(16)}) yields
    $$    
R_3(p)\equiv P_{p^4-p^3-3}(p) \equiv \frac{p^2(p^4-p^3-3)}{2}B_{p^4-p^3-4}\equiv
-\frac{3p^2}{2} B_{p^4-p^3-4}\pmod{p^4}.
 $$
Similarly, if $n=p^4-p^3-4$, then 
 since $p^4-p^3-6\not\equiv 0\,\,(\bmod{\,\, p-1})$,
by  Lemma~\ref{l8}, $p^3\mid p^3B_{p^4-p^3-6}$
for each prime $p\ge 11$. Using this and the fact that $B_{p^4-p^3-5}=0$,
from (\ref{(16)}) modulo $p^3$ we find that
  $$    
R_4(p)\equiv P_{p^4-p^3-4}(p)\equiv pB_{p^4-p^3-4}\pmod{p^3}.
 $$

It remains to show (iv). If $n$ is odd such that 
$n-3\not\equiv 0\,\,(\bmod{\,\, p-1})$,
then by (\ref{(16)}) and Lemma~\ref{l8}, $P_n(p)\equiv\frac{n}{2}p^2B_{n-1}
\,\,(\bmod{\,\, p^4})$ and $P_{n-1}(p)\equiv pB_{n-1}\,\,(\bmod{\,\, p^3})$.
Thus, for such a $n$ we have 
   $$
P_n(p)\equiv \frac{n}{2}pP_{n-1}\pmod{p^4}.
 $$
In particular, for $n=p^4-p^3-5$, from the above we get
 \begin{eqnarray*}
R_5(p)&\equiv& P_{p^4-p^3-5}(p)\equiv\frac{(p^4-p^3-5)p}{2}
P_{p^4-p^3-6}(p)\pmod{p^4}\\
&\equiv&-\frac{5}{2}pP_{p^4-p^3-6}(p)\equiv 
-\frac{5}{2}pR_{6}(p)\pmod{p^4}.
  \end{eqnarray*}
This implies (iv) and the proof is completed.
\end{proof}

\begin{lemma}\label{l13}  For any prime $p$ and any positive
integer $r$, we have
  \begin{equation}\label{(17)}
2R_1\equiv-\sum_{i=1}^rp^iR_{i+1}\pmod{p^{r+1}}.
  \end{equation}
\end{lemma}
\begin{proof}
 Multiplying by $-p/i^2$ ($1\le i\le p-1$)  the identity
  $$
1+\frac{p}{i}+\cdots+\frac{p^{r-1}}{i^{r-1}}=\frac{p^r-i^r}{i^{r-1}(p-i)},
 $$
we obtain
  $$   
-\frac{p}{i^2}\left(1+\frac{p}{i}+\cdots+\frac{p^{r-1}}{i^{r-1}}\right)=
\frac{-p^{r+1}+pi^r}{i^{r+1}(p-i)}
\equiv \frac{p}{i(p-i)}\pmod{p^{r+1}}.
     $$   
Therefore,
 $$
\left(\frac{1}{i}+\frac{1}{p-i}\right)\equiv
-\left(\frac{p}{i^2}+\frac{p^2}{i^3}+\cdots+\frac{p^{r}}{i^{r+1}}\right)
\pmod{p^{r+1}},
 $$
whence after summation $\sum_{i=1}^{p-1}\cdot$, we immediately obtain 
(\ref{(17)}).  
\end{proof}

\begin{proof}[Proof of   Proposition~{\rm\ref{p1}}] 
We begin with the congruence from Propostion 1.
   \begin{equation}\label{(18)}  
{2p-1\choose p-1} \equiv 1 +pR_1
-\frac{p^2}{2}R_2+\frac{p^3}{3}R_3 -
\frac{p^4}{4}R_4+\frac{p^5}{5}R_5-
\frac{p^6}{6}R_6\pmod{p^8}.
  \end{equation}  
As by Lemma~\ref{l4}, $p^2\mid R_7$,
 Lemma~\ref{l13} with $r=7$ yields
 \begin{equation}\label{(19)}
2R_1\equiv -pR_2-p^2R_3-p^3R_4-p^4R_5-p^5R_6\pmod{p^8},
  \end{equation}
whence multiplying by $p/4$ it follows that
 $$
-\frac{p^4}{4}R_4\equiv \frac{p}{2}R_1+\frac{1}{4}
(p^3R_3+p^4R_4+p^5R_5+p^6R_6)
\pmod{p^8}.
 $$
Substituting this into the congruence (\ref{(18)}), we obtain
   $$  
{2p-1\choose p-1} \equiv 1 +\frac{3p}{2}R_1
-\frac{p^2}{4}R_2+\frac{7p^3}{12}R_3
+\frac{9p^5}{20}R_5+\frac{p^6}{12}R_6\pmod{p^8}.
   $$  
Further, from (iv) of Lemma~\ref{l12} we see that
 $$
p^6R_6\equiv -\frac{2}{5}p^5R_5\pmod{p^8}.
 $$
The substitution of this  into the previous congruence
immediately gives
 $$  
{2p-1\choose p-1} \equiv 1 +\frac{3p}{2}R_1
-\frac{p^2}{4}R_2+\frac{7p^3}{12}R_3 
+\frac{5p^5}{12}R_5\pmod{p^8},
   $$  
as desired. 
\end{proof}

\begin{remark} {\rm Proceeding in the same way as in the previous 
proof and using (\ref{(19)}), we can eliminate $R_2$  to obtain}
  $$  
{2p-1\choose p-1} \equiv 1 +2pR_1
+\frac{5p^3}{6}R_3
+\frac{p^4}{4}R_4 
+\frac{17p^5}{30}R_5\pmod{p^8}.
   $$  
\end{remark}

\begin{remark} {\rm If we suppose that there  exists a prime $p$ such that
${2p-1\choose p-1} \equiv 1 \,\,(\bmod{\,\, p^5})$, then by Lemma~\ref{l2}, 
for such a $p$ must be $R_1\equiv 0\,\,(\bmod{\,\, p^4})$ and 
$R_2\equiv 0\,\,(\bmod{\,\, p^3})$. Starting with these two congruences,
in the same manner as in the proof of Lemma~\ref{l7},  
it can be deduced that for $n=2,3,4,5,6,7,8$,   
 $$
H_n\equiv(-1)^{n-1}\frac{R_n}{n}\pmod{p^{e_n}},
  $$
where $e_2=8, e_3=7, e_4=6, e_5=5, e_6=4, e_7=3$ and $e_8=2$.
 Since as in the proof of  Proposition~\ref{p1}, we have
$$
{2p-1\choose p-1}\equiv 1+pH_1+p^2H_2+p^3H_3+p^4H_4
+p^5H_5+p^6H_6+p^7H_7+p^8H_8\pmod{p^{10}},
 $$
then substituting the previous congruences into 
the right hand side of the above congruence
and setting $H_1=R_1$, we obtain
  $$
{2p-1\choose p-1}\equiv 1+pR_1-\frac{p^2}{2}R_2+
\frac{p^3}{3}R_3-\frac{p^4}{4}R_4+
\frac{p^5}{5}R_5-\frac{p^6}{6}R_6
+\frac{p^7}{7}R_7-\frac{p^8}{8}R_8\pmod{p^{10}}.
 $$
Since by Lemma~\ref{l4}, $p^2\mid R_7$ and $p\mid R_8$,
from the above we get
$$
{2p-1\choose p-1}\equiv 1+ pR_1-\frac{p^2}{2}R_2+
\frac{p^3}{3}R_3-\frac{p^4}{4}R_4+
\frac{p^5}{5}R_5-\frac{p^6}{6}R_6
\pmod{p^{9}}.
 $$
Then as in the above proof, using (\ref{(19)}) and the fact that
by (iv) of Lemma~\ref{l12}, 
$p^6R_6(p)\equiv -\frac{2}{5}p^5R_5(p)\,\,(\bmod{\,\, p^9})$,
we can find that}
$$  
{2p-1\choose p-1} \equiv 1 +\frac{3p}{2}R_1
-\frac{p^2}{4}R_2+\frac{7p^3}{12}R_3 
+\frac{5p^5}{12}R_5\pmod{p^9}.
   $$  
\end{remark}  

\begin{proof}[Proof of  Corollary~{\rm\ref{c1}}] In view of the fact that by 
Lemma~\ref{l4}, $p^2\mid R_5$, the congruence from Proposition~\ref{p2}
immediately yields
  \begin{equation}\label{(20)}  
{2p-1\choose p-1} \equiv 1 +\frac{3p}{2}R_1
-\frac{p^2}{4}R_2+\frac{7p^3}{12}R_3 
\pmod{p^7}.
  \end{equation}  
Lemma~\ref{l13} with $r=5$ and the fact that by Lemma~\ref{l4}, 
$p^2\mid R_5$ and $p\mid R_6$ imply  
 $$
2R_1\equiv -pR_2-p^2R_3-p^3R_4\pmod{p^6}.
 $$
From (ii) and (iii) of Lemma~\ref{l12} we see 
that $pR_4\equiv-\frac{2}{3}R_3\,\,(\bmod \,\, p^4)$,
so  that $p^3R_4\equiv-\frac{2}{3}p^2R_3\,\,(\bmod \,\, p^6)$.
Substituting this into the previous congruence, we obtain
  $$
 2R_1+pR_2\equiv-\frac{1}{3}p^2R_3\pmod{p^6},
 $$
whence we have 
 \begin{equation}\label{(21)}
p^3R_3\equiv -6pR_1-3p^2R_2 \pmod{p^7}.
  \end{equation}
Substituting this into (\ref{(20)}), we get
$$  
{2p-1\choose p-1} \equiv 1 -2pR_1
-2p^2R_2\pmod{p^7},
  $$  
that is actually the first congruence
from Corollary~\ref{c1}. Finally, from (\ref{(21)})
we have 
 $$
p^2R_2\equiv -2pR_1-\frac{1}{3}p^3R_3 \pmod{p^7},
 $$
which replacing into (\ref{(20)}) gives
 \begin{equation}\label{(22)}  
{2p-1\choose p-1} \equiv 1 +2pR_1
+\frac{2}{3}p^3R_3\pmod{p^7}.
  \end{equation}  
This completes the proof.
\end{proof}

\begin{proof}[Proof of  Corollary~{\rm\ref{c2}}]
 By (ii) of Lemma~\ref{l12}, we have
$p^3R_3(p)\equiv -\frac{3}{2}p^5B_{p^4-p^3-4}\pmod{p^7}$.
Substituting this into (\ref{(22)}), we obtain
 \begin{equation}\label{(23)}  
{2p-1\choose p-1} \equiv 1 +2pR_1-p^5B_{p^4-p^3-4}\pmod{p^7}.
  \end{equation}
By Lemma~\ref{l3}, $p\mid B_{p-3}$, so that  $p^6\mid \frac{p^5}{6}B_{p-3}$, and 
hence, from  (i) of Lemma~\ref{l12} we obtain
  $$
2pR_1(p)\equiv -p^3B_{p^4-p^3-2}
-\frac{p^5}{2}B_{p^2-p-4}+\frac{p^6}{10}B_{p-5}\pmod{p^7}.
 $$
Furthermore,  by the Kummer congruences (\ref{(10)}),
since $p^4-p^3-2\not\equiv 0\,\,(\bmod{\,\, p-1})$
and $p^4-p^3-2\equiv p^2-p-2
\,\,(\bmod{\,\,\varphi(p^2)})$, we have
  $$
B_{p^4-p^3-4}\equiv
\frac{p^4-p^3-4}{p^2-p-4}B_{p^2-p-4}\equiv
\frac{4}{p+4}B_{p^2-p-4}\equiv
\left(1-\frac{p}{4}\right)B_{p^2-p-4}\pmod{p^2}.
  $$
The substitution of the above two congruences into (\ref{(22)}) 
immediately gives
 \begin{equation}\label{(24)}  
{2p-1\choose p-1} \equiv 1-p^3B_{p^4-p^3-2}
-\frac{3p^5}{2}B_{p^2-p-4}+\frac{p^6}{10}B_{p-5}+
\frac{p^6}{4}B_{p^2-p-4}\pmod{p^7}.
 \end{equation}
Finally, since by the Kummer congruences (\ref{(10)}),
 $$
B_{p^2-p-4}\equiv\frac{p^2-p-4}{p-5}B_{p-5}\equiv
\frac{4}{5}B_{p-5}\pmod{p},
 $$
after substitution of this into (\ref{(24)}) we obtain
 \begin{equation}\label{(25)}  
{2p-1\choose p-1} \equiv 1-p^3B_{p^4-p^3-2}
-\frac{3p^5}{2}B_{p^2-p-4}+\frac{3p^6}{10}B_{p-5}\pmod{p^7}.
  \end{equation}
This is the required congruence.
\end{proof}

\begin{proof}[Proof of  Corollary~{\rm\ref{c3}}]  As noticed in 
(\cite[the congruence (3) on page
494]{HT}), combining the Kummer congruences (\ref{(10)}) and (\ref{(11)}) for 
$m=\varphi(p^n)-s$
($n,s\in\mathbf N$ with $s\not\equiv 0\,\,(\bmod{\,\, p-1})$,
 we obtain
  \begin{equation}\label{(26)}
\frac{B_{p^n-p^{n-1}-s}}{p^n-p^{n-1}-s}\equiv 
\sum_{k=1}^n(-1)^{k+1}{n\choose k}
\frac{B_{k(p-1)-s}}{k(p-1)-s}\pmod{p^n}.
  \end{equation}
Now (\ref{(26)}) with $n=2$ and $s=4$ gives
 $$
\frac{B_{p^2-p-4}}{p^2-p-4}\equiv \frac{2B_{p-5}}{p-5}-
\frac{B_{2p-6}}{2p-6}\pmod{p^2},
 $$
or equivalently
 $$
B_{p^2-p-4}\equiv -\frac{2(p+4)}{p-5}B_{p-5}+
\frac{2(p+4)}{p-3}B_{2p-6}\pmod{p^2}.
 $$
Substituting $1/(p-5)\equiv -(5+p)/25 \,\,(\bmod{\,\,p^2})$ and
$1/(p-3)\equiv -(3+p)/9 \,\,(\bmod{\,\,p^2})$,
the above congruence becomes
  \begin{equation}\label{(27)}
B_{p^2-p-4}\equiv\frac{18p+40}{25}B_{p-5}-
\frac{7p+12}{18}B_{2p-6}\pmod{p^2}.
  \end{equation}
Similarly, (\ref{(26)}) with $n=4$ and $s=2$ yields
 $$
\frac{B_{p^4-p^3-2}}{p^4-p^3-2}\equiv\sum_{k=1}^4(-1)^{k+1}{4\choose k}
\frac{B_{k(p-1)-2}}{k(p-1)-2} \pmod{p^4},
 $$
whence multiplying by $p^3+2$,
we get
   \begin{equation}\label{(28)}\begin{split}
-B_{p^4-p^3-2}
\equiv & (p^3+2)\left(\frac{4B_{p-3}}{p-3}-
\frac{6B_{2p-4}}{2p-4}+\frac{4B_{3p-5}}{3p-5}-
\frac{B_{4p-6}}{4p-6}\right)\pmod{p^4}\\
 \equiv & p^3\left(\frac{4B_{p-3}}{-3}-
\frac{6B_{2p-4}}{-4}+\frac{4B_{3p-5}}{-5}-
\frac{B_{4p-6}}{-6}\right)\\
&+2\left(\frac{4B_{p-3}}{p-3}-
\frac{6B_{2p-4}}{2p-4}+\frac{4B_{3p-5}}{3p-5}-
\frac{B_{4p-6}}{4p-6}\right)\pmod{p^4}.
  \end{split}\end{equation}
As by the Kummer congruences (\ref{(10)}),
   $$
\frac{B_{4p-6}}{4p-6}\equiv\frac{B_{3p-5}}{3p-5}  \equiv
\frac{B_{2p-4}}{2p-4}\equiv\frac{B_{p-3}}{p-3}\pmod{p},
  $$
we have 
$$
B_{4p-6}\equiv 2B_{p-3}\,\,(\bmod{\,\,p}),
B_{3p-5}\equiv \frac{5}{3} B_{p-3}\,\,(\bmod{\,\,p}),
B_{2p-4}\equiv \frac{4}{3}\cdot B_{p-3}\,\,(\bmod{\,\,p}).
 $$
Substituting this into the first term on the right-hand side in
the congruence (\ref{(28)}), we obtain
 $$
p^3\left(\frac{4B_{p-3}}{-3}-
\frac{6B_{2p-4}}{-4}+\frac{4B_{3p-5}}{-5}-
\frac{B_{4p-6}}{-6}\right)\equiv
-\frac{p^3}{3}B_{p-3}\equiv 0\pmod{p^4},
  $$
where we have  use the fact that by Lemma~\ref{l3}, 
$p$ divides the numerator of $B_{p-3}$.

Further, as for every integers $a,b,n$ such that
$b\not\equiv 0\,\,(\bmod{\,\, p})$, holds 
 $$
\frac{1}{ap-b}\equiv -
\frac{1}{b}\sum_{k=0}^{3}\frac{a^kp^k}{b^k}\pmod{p^4},
 $$
applying this to $1/(p-3)$,  $1/(2p-4)$ and $1/(3p-5)$,  
the second term on the right-hand side in
the  congruence (\ref{(28)}) becomes
\begin{eqnarray*}
-B_{p^4-p^3-2}&\equiv&2\left(-\frac{4}{3}\big(1+\frac{p}{3}+\frac{p^2}{9}\big)
B_{p-3}+\frac{3}{2}\big(1+\frac{p}{2}+\frac{p^2}{4}\big)B_{2p-4}\right.   \\
&&-\left. \frac{4}{5}\big(1+\frac{3p}{5}+\frac{9p^2}{25}\big)B_{3p-5}+
\frac{1}{6}\big(1+\frac{2p}{3}+\frac{4p^2}{9}\big)B_{4p-6}\right)\pmod{p^4}.
  \end{eqnarray*}
Muptiplying by $p^3$, the above congruence becomes
\begin{eqnarray*}
-p^3B_{p^4-p^3-2}&\equiv&-\frac{8}{3}\big(p^3+\frac{p^4}{3}+\frac{p^5}{9}\big)
B_{p-3}+3\big(p^3+\frac{p^4}{2}+\frac{p^5}{4}\big)B_{2p-4}   \\
&&- \frac{8}{5}\big(p^3+\frac{3p^4}{5}+\frac{9p^5}{25}\big)B_{3p-5}+
\frac{1}{3}\big(p^3+\frac{2p^4}{3}+\frac{4p^5}{9}\big)B_{4p-6}\pmod{p^7}.
  \end{eqnarray*}
Finally, substituting the above congruence and the congruence
(\ref{(27)}) into (\ref{(25)}), we obtain the congruence 
from Corollary 3.
\end{proof}


\begin{thebibliography}{9999} 

\bibitem[B]{B} M. B{\scriptsize AYAT},  A generalization of 
Wolstenholme's theorem, {\it Amer. Math. Monthly} {\bf 104} (1997), 557--560. 

\bibitem[Gl1]{Gl1} J.W.L. G{\scriptsize LAISHER},  
Congruences relating to the sums of products of the first 
$n$ numbers and to other sums of products,
{\it Q. J. Math.} {\bf 31} (1900), 1--35. 


\bibitem[Gl2]{Gl2} J.W.L. G{\scriptsize LAISHER},  
On the residues of the sums  of products of the first $p-1$
numbers and their powers, to modulus $p^2$ or $p^3$,
{\it Q.  J. Math.} {\bf 31} (1900) 321--353. 

\bibitem[Gr]{Gr} A. G{\scriptsize RANVILLE},  
Arithmetic properties of binomial coefficients. I.
Binomial coefficients modulo prime powers, in {\it Organic Mathematics-
$($Burnaby, BC $1995$}, CMS Conf. Proc., vol. 20, American Mathematical 
Society, Providence, RI, 1997, 253--276.



\bibitem[HT]{HT} C. H{\scriptsize ELOU} and 
G. T{\scriptsize ERJANIAN}, On Wolstenholme's theorem and
its converse, {\it J. Number Theory} {\bf 128} (2008), 475--499. 
 

\bibitem[HW]{HW} G.H. H{\scriptsize ARDY} and  E.M. W{\scriptsize RIGHT},
{\it An Introduction to the Theory of Numbers}, Clarendon Press, Oxford, 1980.

\bibitem[IR]{IR} K. I{\scriptsize RELAND} and M. R{\scriptsize OSEN},
{\it A Classical Introduction to Modern Number Theory},
Springer-Verlag, New York, 1982.


\bibitem[J]{J} N. J{\scriptsize ACOBSON},
{\it Basic Algebra I, 2nd Edition}, W.H. Freeman Publishing Company,
New York, 1995.


\bibitem[K]{K} E. E. K{\scriptsize UMMER},  \"{U}ber eine allgemeine 
Eigenschaft der rationalen Entwicklungsco\"{e}fficienten einer 
bestimmten Gattung analytischer Functionen, {\it J. Reine Angew. Math.}, 
{\bf 41} (1851), 368--372.   


\bibitem[L]{L} E. L{\scriptsize EHMER}, On congruences 
involving Bernoulli numbers and the quotients of
Fermat and Wilson, {\it Ann. Math.} {\bf 39} (1938), 350--360. 


\bibitem[M]{M} R.J. M{\scriptsize C}I{\scriptsize NTOSH},  On the converse of 
 Wolstenholme's Theorem, {\it Acta Arith.}  {\bf 71} (1995), 381--389. 


\bibitem[MR]{MR} R.J. M{\scriptsize C}I{\scriptsize NTOSH} and E.L. 
R{\scriptsize OETTGER}, A search for 
Fibonacci-Wieferich and Wolstenholme primes, {\it Math. Comp.} 
{\bf 76} (2007), 2087--2094. 


\bibitem[Me]{Me} R. M{\scriptsize  E\v{S}TROVI\'C}, 
On the  mod $p^7$ determination of ${2p-1\choose p-1}$, 
 preprint {\tt arXiv: 1108.1174v1 [math.NT]} (2011).


\bibitem[SW]{SW} Z.W.  S{\scriptsize UN} and D. W{\scriptsize AN}, 
On Fleck quotients, {\it Acta Arith.} {\bf 127} (2007), 337--363.  


\bibitem[W]{W} J. W{\scriptsize OLSTENHOLME}, On certain properties
of prime numbers, {\it Quart. J. Pure Appl. Math.} {\bf 5} (1862), 35--39. 


\bibitem[Z1]{Z1} J. Z{\scriptsize HAO}, {\it Wolstenholme type theorem for 
multiple harmonic sums},  Int. J. Number Theory {\bf 4} (2008), 73--106.

\bibitem[Z2]{Z2} J. Z{\scriptsize HAO}, Bernoulli Numbers,  Wolstenholme's 
theorem, and $p^5$ variations of Lucas' theorem, 
{\it J. Number Theory} {\bf 123} (2007), 18--26.  


\bibitem[ZC]{ZC} X. Z{\scriptsize HOU} and T. C{\scriptsize AI}, 
A generalization of a curious congruence on
harmonic sums, {\it Proc. Amer. Math. Soc.} {\bf 135} (2007), 1329--1333.  
 
\end{thebibliography}
\end{document}